\documentclass{amsart}
\usepackage{amsmath,amssymb,amsthm}
\usepackage{enumerate}

\newtheorem{lemma}{Lemma}
\newtheorem{thm}{Theorem}
\newtheorem{prop}{Proposition}
\newtheorem{cor}{Corollary}

\newcommand{\dbarb}{\overline{\partial}_b}
\newcommand{\boxb}{\square_b}
\newcommand{\Zbar}{\overline{Z}}

\begin{document}

\title[Sobolev inequalities for $(0,q)$ forms on CR manifolds]{Sobolev Inequalities for $(0,q)$ forms\\ on CR manifolds of finite type}
\author{Po-Lam Yung}
  \address{Department of Mathematics \\ Princeton University \\ NJ 08544}
  \email{pyung@math.princeton.edu}

\begin{abstract}
Let $M^{2n+1}$ ($n \geq 2$) be a compact pseudoconvex CR manifold of finite commutator type whose $\dbarb$ has closed range in $L^2$ and whose Levi form has comparable eigenvalues. We prove a Gagliardo-Nirenberg inequality for the $\dbarb$ complex for $(0,q)$ forms when $q \ne 1$ nor $n-1$. We also prove an analogous inequality when $M$ satisfies condition $Y(q)$. The main technical ingredient is a new kind of $L^1$ duality inequality for vector fields that satisfy Hormander's condition. 
\end{abstract}

\maketitle

\section{Introduction}

Recently Bourgain-Brezis \cite{MR2293957} and Lanzani-Stein \cite{MR793239} established the following $L^1$ Sobolev inequality (or Gagliardo-Nirenberg inequality) for differential forms: If $u$ is a smooth compactly supported $q$ form on $\mathbb{R}^n$ and $q \ne 1$ nor $n-1$, then 
$$
\|u\|_{L^{\frac{n}{n-1}}(\mathbb{R}^n)} \lesssim \|du\|_{L^1(\mathbb{R}^n)} + \|d^* u\|_{L^1(\mathbb{R}^n)}
$$ 
where $d$ is the Hodge de-Rham differential operator and $d^*$ is its adjoint under the flat Euclidean metric. This generalizes the classical Gagliardo-Nirenberg inequality for functions, which corresponds to the case $q=0$ or $n$. This is, however, a very remarkable inequality when $2 \leq q \leq n-2$, for while the corresponding inequality when $L^1$ is replaced by $L^p$ ($1 < p < n$) follows easily from the classical Calderon-Zygmund theory of singular integrals, the Calderon-Zygmund theory breaks down for $L^1$. In fact a simple example shows that the inequality is false when $q = 1$ or $n-1$ and $n \geq 2$. At the heart of this is a new kind of $L^1$ duality inequality, which says that if $f = (f_1, \dots, f_n)$ is a compactly supported smooth vector field on $\mathbb{R}^n$ and $\text{div} f = g$, 
then for any $\Phi \in C^{\infty}_c(\mathbb{R}^n)$, 
\begin{equation}\label{eq:Eucldivcurl}
\left|\int_{\mathbb{R}^n} f_1(x) \Phi(x) dx \right| \lesssim \|f\|_{L^1} \|\nabla \Phi\|_{L^n} + \|g\|_{L^1} \|\Phi\|_{L^n}.
\end{equation}
(See van Schaftingen \cite{MR2078071}.) This can be thought of as a remedy to the failure of the Sobolev embedding of $W^{1,n}$ into $L^{\infty}$, for if the embedding holds, the inequality would become trivial. Note also that this inequality does not follow from classical compensated compactness arguments.

More recently, Chanillo and van Schaftingen \cite{MR2511628} proved an analog of this inequality on a general homogeneous group: If $G$ is a homogeneous group of homogeneous dimension $Q$ and $X_1, \dots, X_n$ is a basis of left-invariant vector fields of degree 1 on $G$, then for any functions $f_1, \dots, f_n, g \in C^{\infty}_c(G)$ which satisfies $X_1 f_1 + \dots + X_n f_n = g$ and any $\Phi \in C^{\infty}_c(G)$, we have 
$$
\left|\int_{G} f_1(x) \Phi(x) dx \right| \lesssim \|f\|_{L^1(G)} \|\nabla_b \Phi\|_{L^Q(G)} + \|g\|_{L^1(G)} \|\Phi\|_{L^Q(G)} 
$$
where $\nabla_b \Phi = (X_1 \Phi, \dots, X_n \Phi)$. Our first result generalizes this:

\begin{thm}\label{thm:PLA}
Let $X_1, \dots, X_n$ be smooth real vector fields in a neighborhood of $0$ in $\mathbb{R}^N$. Suppose they are linearly independent at 0 and their commutators of length $\leq r$ span the tangent space at $0$. Let $V_j(x)$ be the span of the commutators of $X_1, \dots, X_n$ of length $\leq j$ at $x$, and let $Q$ be defined by 
$$
Q := \sum_{j=1}^r j n_j, \qquad n_j := \dim V_j(0) - \dim V_{j-1}(0).
$$
Then there exists a neighborhood $U$ of $0$ and $C > 0$ such that if 
$$
X_1 f_1 + \dots + X_n f_n = g
$$ 
on $U$ with $f_1, \dots, f_n, g \in C^{\infty}_c(U)$ and $\Phi \in C^{\infty}_c (U)$, then 
$$
\left|\int_{U} f_1(x) \Phi(x) dx \right| \leq C \left(\|f\|_{L^1(U)} \|\Phi\|_{NL_1^Q(U)} + \|g\|_{L^1(U)} \|\Phi\|_{L^Q(U)}\right)
$$ 
where $\|\Phi\|_{NL_1^Q(U)} = \|\nabla_b \Phi\|_{L^Q(U)} + \|\Phi\|_{L^Q(U)},$ and $\nabla_b \Phi = (X_1 \Phi, \dots, X_n \Phi)$.
\end{thm}

Theorem~\ref{thm:PLA} allows us to study the $\dbarb$ complex of two classes of CR manifolds of finite commutator type, and prove a Gagliardo-Nirenberg inequality for $(0,q)$ forms that involves the $\dbarb$ complex. A CR manifold $M$ is said to be of finite commutator type $m$ at a point $x$ if the brackets of real and imaginary parts of the (1,0) vector fields of length $\leq m$ span the tangent space of $M$ at $x$; and a pseudoconvex CR manifold $M^{2n+1}$ is said to satisfy condition $D(q)$ if there is a constant $C > 0$ such that for any point $x \in M$, the sum of any $q$ eigenvalues of the Levi form at $x$ is bounded by $C$ times any other such sum, for $1 \leq q \leq n/2$.
The condition $D(1)$ is usually loosely referred to as that $M$ has comparable Levi eigenvalues, because this condition is simply that for some $C > 0$, for any $x \in M$ and any eigenvalues $\lambda_1(x)$, $\lambda_2(x)$ of the Levi form at $x$, we have $\lambda_1(x) \leq C \lambda_2(x)$. 

\begin{thm}\label{thm:PLB}
Let $M$ be a compact orientable pseudoconvex CR manifold of real dimension $2n+1 \geq 5$, for which the range of $\dbarb$ on $(0,q)$ forms is closed in $L^2$ for all $q$. Suppose that
\begin{enumerate}[(i)]
\item $M$ is of finite commutator type $m$ at every point, and
\item $M$ satisfies condition $D(q_0)$ for some $1 \leq q_0 \leq n/2$.
\end{enumerate}
Let $Q = 2n+m.$ Then 
\begin{enumerate}[(a)]
\item
If $q_0 \leq q \leq n-q_0$ and $q \ne 1$ nor $n-1$, then for any smooth $(0,q)$ form $u$ on $M$ that is orthogonal to the kernel of $\boxb$, we have
$$
\|u\|_{L^{\frac{Q}{Q-1}}(M)} \lesssim \|\dbarb u\|_{L^1(M)} + \|\dbarb^* u\|_{L^1(M)}.
$$
\item
For any smooth $(0,q_0-1)$ form $v$ orthogonal to the kernel of $\dbarb$, we have 
$$
\|v\|_{L^{\frac{Q}{Q-1}}(M)} \lesssim \|\dbarb v\|_{L^1(M)}.
$$
\item
For any smooth $(0,n-q_0+1)$ form $w$ orthogonal to the kernel of $\dbarb^*$, we have 
$$
\|w\|_{L^{\frac{Q}{Q-1}}(M)} \lesssim \|\dbarb^* w\|_{L^1(M)}.
$$
\end{enumerate}
\end{thm}

In particular when $q_0 = 1$, i.e. when $M$ has comparable Levi eigenvalues, then $$\|v\|_{L^{\frac{Q}{Q-1}}(M)} \lesssim \|\dbarb v \|_{L^1(M)}$$ for any function $v$ orthogonal to the kernel of $\dbarb$, which can be thought of as a Gagliardo-Nirenberg inequality for $\dbarb$. Also, in this case 
$$\|u\|_{L^{\frac{Q}{Q-1}}(M)} \lesssim \|\dbarb u\|_{L^1(M)} + \|\dbarb^* u\|_{L^1(M)}$$ 
for any smooth $(0,q)$ forms $u$ orthogonal to the kernel of $\boxb$, when $q \ne 1$ nor $n-1$.

Finally, a CR manifold $M^{2n+1}$ is said to satisfy condition $Y(q)$ if at every point the Levi form has $\max(q+1,n-q+1)$ eigenvalues of the same sign or $\min(q+1,n-q+1)$ pairs of eigenvalues of opposite signs. Note that all such manifolds are necessarily of finite commutator type 2. 

\begin{thm}\label{thm:PLC}
Let $M^{2n+1}$ be a compact orientable CR manifold that satisfies condition $Y(q)$ for some $0 \leq q \leq n$, and $Q = 2n+2$. 
\begin{enumerate}[(a)]
\item If $q \ne 1$ nor $n-1$, and $u$ is a smooth $(0,q)$ form orthogonal to the kernel of $\boxb$, $$\|u\|_{L^{\frac{Q}{Q-1}}(M)} \lesssim \|\dbarb u\|_{L^1(M)} + \|\dbarb^* u\|_{L^1(M)}.$$
\item If $q \ne 0$ nor $n$, and $v$ is a smooth $(0,q-1)$ form orthogonal to the kernel of $\dbarb$, $$\|v\|_{L^{\frac{Q}{Q-1}}(M)} \lesssim \|\dbarb v\|_{L^1(M)}.$$
\item If $q \ne 0$ nor $n$, and $w$ is a smooth $(0,q+1)$ form orthogonal to the kernel of $\dbarb^*$, $$\|w\|_{L^{\frac{Q}{Q-1}}(M)} \lesssim \|\dbarb^* w\|_{L^1(M)}.$$
\end{enumerate}
\end{thm}

The case of strongly pseudoconvex CR manifolds of dimension $2n+1 \geq 5$ is covered under both Theorems~\ref{thm:PLB} and~\ref{thm:PLC}, with $Q = 2n+2$. For instance,

\begin{cor}
If $M^{2n+1}$ is a compact orientable strongly pseudoconvex CR manifold of dimension $2n+1 \geq 5$ and $q \ne 1$ nor $n-1$, then for any smooth $(0,q)$ form $u$ orthogonal to the kernel of $\boxb$, we have
$$\|u\|_{L^{\frac{Q}{Q-1}}(M)} \lesssim \|\dbarb u\|_{L^1(M)} + \|\dbarb^* u\|_{L^1(M)}$$
where $Q=2n+2$.
\end{cor}

A few remarks are in order. First, part of the difficulty in Theorem~\ref{thm:PLA} is in proving the inequality with the best possible value of $Q$; it can be shown, using local dilation invariance, that the inequality in Theorem~\ref{thm:PLA} cannot hold for any value of $Q$ smaller than the one that is given there. Hence the value of $Q$ as defined in Theorem~\ref{thm:PLA} should be thought of as the correct \emph{non-isotropic dimension} that one should attach to the point 0 in such a situation.

Note also that with $Q$ as given in Theorem~\ref{thm:PLA}, there is a Sobolev inequality for functions $u$ that satisfies $u, \nabla_b u \in L^p$ if $1 \leq p < Q$ (see Proposition~\ref{prop:SE} below, which we state without proof). Theorem~\ref{thm:PLA} can be taken as a remedy of the failure of this embedding when $p = Q$.

\begin{prop}\label{prop:SE}
Let $X_1, \dots, X_n$ be smooth real vector fields on $\mathbb{R}^N$, whose commutators of length $\leq r$ span at $0$. Let $Q$ be the non-isotropic dimension at $0$ as defined in Theorem~\ref{thm:PLA}. Then there exists a neighborhood $U$ of $0$ and $C > 0$ such that if $u \in C^{\infty}_c(U)$ and $1 \leq p < Q$, then 
$$
\|u\|_{L^{p^*}(U)} \leq C \left(\|\nabla_b u\|_{L^p(U)} + \|u\|_{L^p(U)} \right) \quad \text{where} \quad \frac{1}{p^*} = \frac{1}{p} - \frac{1}{Q}.
$$
Moreover the inequality cannot hold for any bigger value of $p^*$.
\end{prop}

We remark that Capogna, Danielli, and Garofalo have obtained a similar Sobolev inequality in \cite{MR1266765}, but our proposition is sharper in general because we are always using the best (i.e. smallest) possible value of $Q$ in the definition of $p^*$. See also the work of Varopoulos \cite{MR1070036} and Gromov \cite[Section 2.3.D'']{MR1421823}.

Next in Theorem~\ref{thm:PLB}, the assumption that the ranges of $\dbarb$ on $(0,q)$ forms are closed in $L^2$ for all $q$ is met under fairly general conditions; by the results of Kohn \cite{MR850548} and Nicoara \cite{MR2189215}, this assumption is satisfied by all boundaries of bounded weakly pseudoconvex domains in $\mathbb{C}^{n+1}$, and more generally by all embeddable compact orientable CR manifolds of dimension $\geq 5$. The assumption of comparable sums of eigenvalues was made to ensure that maximal subellipticity holds in the $L^p$ sense (see Koenig \cite{MR1879002}). We made the assumption that $M$ has real dimension $2n+1 \geq 5$ because if the real dimension of $M$ were $2n+1=3$, then $n=1$, in which case $\dbarb$ produces only top forms and $\dbarb^*$ produces only functions. In these cases our method does not say anything about $(0,q)$ forms on $M$ for any $q$.

Finally, in effect Theorem \ref{thm:PLC} also only applies to CR manifolds of dimension $2n+1 \geq 5$, because condition $Y(q)$ is never satisfied in 3 dimensions for any $q$. Moreover, since conditions $Y(q)$ and $Y(n-q)$ are equivalent, one can formulate the corresponding inequalities for $(0,n-q)$ forms, $(0,n-q-1)$ forms and $(0,n-q+1)$ forms.

The main new ingredient in the proof of Theorem~\ref{thm:PLA} is that it involves a lifting of the given vector fields $X_1, \dots, X_n$ to a higher dimensional Euclidean space where they can be approximated by left-invariant vector fields of a homogeneous group. This approximation is crucial in making certain integration by parts argument work when we construct and estimate some convolution-like integrals. That such a lifting is possible was shown in the work of Rothschild and Stein \cite{MR0436223}. One can then adapt some previously known arguments, as in \cite{MR2078071}, \cite{MR2122730} and \cite{MR2511628}, to prove Theorem~\ref{thm:PLA}. The new challenge here is to still bound things in $L^Q$ with the correct value of $Q$ (as designated in Theorem~\ref{thm:PLA}) despite of the lifting (because lifting introduces a new space with a bigger non-isotropic dimension). This is done by carefully integrating out the added variables. Some lower order errors that arise from the approximation also need to be taken care of. The technical contents are contained in the proof of Lemma~\ref{lem:decomp} below.

It will be of interest to see in Theorem~\ref{thm:PLA} whether the assumption of linear independence of $X_1, \dots, X_n$ at $0$ can be replaced by some other weaker non-degeneracy conditions, although in our study of the $\dbarb$ complex this assumption is always satisfied. In fact a large part of our argument, namely Lemma~\ref{lem:decomp} below, goes through without having to assume this linear independence. It is only in the final argument of the proof of Theorem~\ref{thm:PLA} that we need that. 

\section{$L^1$ duality inequality for Hormander's vector fields}\label{sect:div-curl}

In this section we shall prove Theorem~\ref{thm:PLA}. The proof has its starting point the argument of van Schaftingen \cite{MR2078071}, Lanzani-Stein \cite{MR2122730} and Chanillo-van Schaftingen \cite{MR2511628} that freezes one of the variables in the integral to be estimated. We shall also need a variant of a decomposition lemma that appeared in their work (see Lemma~\ref{lem:decomp} below).

First observe that the Lebesgue measure $dx$ was used in the statement of Theorem~\ref{thm:PLA}, but the Lebesgue measure on $U$ depends on the choice of a coordinate system $x$. In proving Theorem~\ref{thm:PLA}, however, we are free to choose any coordinate system $x$ on $U$, because if the inequality holds in one coordinate system, then it holds in any other coordinate system. This is because the Lebesgue measure in one coordinate system is just the Lebesgue measure in another multiplied by a smooth function. Below we shall choose some `normal coordinate system' $x$ using the given vector fields, and  prove the inequality in that coordinate system. 

Next let $X_1, \dots, X_n$ be smooth real vector fields in $\mathbb{R}^N$, whose commutators of length $\leq r$ span the tangent space at $0$. We shall not require them to be linearly independent at $0$ except in the proof of Theorem~\ref{thm:PLA} below. Let $\{X_{jk}\}_{1 \leq j \leq r, 1 \leq k \leq n_j}$ be a collection of vector fields that satisfies the following:
\begin{enumerate}[(a)]
 \item Each $X_{jk}$ is a commutator of $X_1, \dots, X_n$ of length $j$;
 \item For each $1 \leq j_0 \leq r$, $\{X_{jk}\}_{1 \leq j \leq j_0, 1 \leq k \leq n_j}$ restricts at $0$ to a basis of $V_{j_0}(0)$.
\end{enumerate}
Without loss of generality we assume $X_{1k} = X_k$ for all $1 \leq k \leq n_1$. (Note that we must have $n_1 \geq 1$ for the commutators of $X_1, \dots, X_n$ to span at $0$.)
Given a point $\xi$ and a vector field $X$, we shall write $\exp(X)\xi$ for the time-1-flow along the integral curve of $X$ beginning at $\xi$. Then for each point $\xi$ near $0$, 
$$
x \mapsto \exp(x \cdot X') \xi, \qquad x \cdot X' := \sum_{j = 1}^r \sum_{k=1}^{n_j} x_{jk} X_{jk}
$$ 
defines a normal coordinate system locally near $\xi$,
where $x = (x_{jk})_{1 \leq j \leq r, 1 \leq k \leq n_j}$. Throughout we shall take $U$ to be a (sufficiently small) totally normal neighborhood of $0$, which means that it is a normal neighborhood of each of its points. Since we have already restricted ourselves to consider only functions that have compact support in $U$, we shall use consistently identify $U$ as a subset of the tangent space $T_0(\mathbb{R}^N)$ of $\mathbb{R}^N$ at $0$ using the exponential map. In particular, we shall consistently write $x$ for $\exp(x \cdot X')0$. Hence $x$ shall denote the normal coordinates of $U$ at $0$. Any compactly supported function on $U$ will automatically be extended to $T_0(\mathbb{R}^N)$ by 0 outside $U$. We shall often just write $\mathbb{R}^N$ for $T_0(\mathbb{R}^N)$.

The following decomposition lemma is a generalization of the key lemma in Chanillo and van Schaftingen \cite{MR2511628}.

\begin{lemma}\label{lem:decomp}
Let $U$ be a sufficiently small totally normal neighborhood of $0$ and $I$ be the set of all $a \in \mathbb{R}$ for which $\{x_{11}=a\}\cap U \ne \emptyset$. Then for any $\Phi \in C^{\infty}_c(U)$, any $a \in I$ and any $\lambda > 0$, there is a decomposition of the restriction of $\Phi$ to the hyperplane $\{x_{11}=a\}\cap U$ into $$\left.\Phi\right|_{\{x_{11}=a\}\cap U} = \Phi_1^a + \Phi_2^a$$ and an extension of $\Phi_2^a$ to the whole $U$ (which we still denote by $\Phi_2^a$) such that $\Phi_2^a \in C^{\infty}(U)$ and 
\begin{align*}
\|\Phi_1^a\|_{L^{\infty}(\{x_{11}=a\} \cap U)} &\leq C \lambda^{\frac{1}{Q}} M\mathcal{I}(a)\\
\|\nabla_b \Phi_2^a\|_{L^{\infty}(U)} &\leq C \lambda^{\frac{1}{Q}-1} M\mathcal{I}(a)\\
\|\Phi_2^a\|_{L^{\infty}(U)} &\leq C \lambda^{\frac{1}{Q}-1} M\mathcal{J}(a)
\end{align*}
where 
$$\begin{cases}
\mathcal{I}(x_{11}) = \left(\int_{\overline{x} \in \mathbb{R}^{N-1}} (|\nabla_b \Phi|^Q + |\Phi|^Q) (x) d\overline{x}\right)^{\frac{1}{Q}} \\
\mathcal{J}(x_{11}) = \left(\int_{\overline{x} \in \mathbb{R}^{N-1}} |\Phi|^Q (x) d\overline{x}\right)^{\frac{1}{Q}} 
\end{cases},
\quad x = (x_{11},\overline{x}),$$
$d\overline{x} = dx_{12} \dots dx_{rn_r}$ and $M$ is the standard Hardy-Littlewood maximal function on $\mathbb{R}$.
\end{lemma}

Assuming the lemma for the moment, we shall now adopt the argument of van Schaftingen \cite{MR2078071} to finish the proof of Theorem~\ref{thm:PLA}. The main difficulty is that now when one freeze say the $x_{11}$ the coefficient, the vector fields $X_2, \dots, X_n$ are no longer tangent to the hyperplanes where $x_{11}$ is constant. This would kill the whole integration by parts argument by introducing extra boundary integrals that one cannot control. Fortunately, when the vector fields $X_1, \dots, X_n$ are linearly independent at $0$ and $U$ is sufficiently small, the transverse components of $X_2, \dots, X_n$ to the hyperplanes $\{x_{11} = \text{constants}\}$ are small near $0$, and a perturbation argument would then work. 

\begin{proof}[Proof of Theorem~\ref{thm:PLA}]
Let $U$ be a small neighborhood of $0$ on which Lemma~\ref{lem:decomp} holds. Shrinking $U$ if necessary, we may assume that for $1 \leq k \leq n$, $X_k$ is transverse to all hyperplanes $\{x_{1k}=a\}$ that intersect $U$. For $l \ne k$, decompose $X_l$ into $$X_l = X_{l}^k + a_{kl}(x) X_k$$ where $X_{l}^k$ are parallel to all the hyperplanes $\{x_{1k}=a\}$ that intersect $U$ and $a_{kl}(x)$ are smooth functions of $x$ with $a_{kl}(0)=0$. By further shrinking $U$ if necessary we may assume all $\|a_{kl}\|_{L^{\infty}(U)}$ are sufficiently small.

Suppose $X_1f_1 + \dots + X_n f_n = g$ in $U$, where $f_1,\dots,f_n,g$ are all in $C^{\infty}_c(U)$. 
Then
\begin{equation}\label{eq:divcurlX1}
X_1 F_1 + X_2^1 f_2 + \dots + X_n^1 f_n = X_1(a_{12}) f_2 + \dots + X_1(a_{1n}) f_n + g
\end{equation}
where $$F_1 = f_1 + a_{12}f_2 + \dots + a_{1n}f_n.$$ We shall show that 
$$\left|\int_{U} F_1(x) \Phi(x) dx\right| \leq C \left(\|f\|_{L^1(U)} \|\Phi\|_{NL_1^Q(U)} + \|g\|_{L^1(U)} \|\Phi\|_{L^Q(U)}\right)$$
for $\Phi \in C^{\infty}_c(U)$. 
Assuming this for the moment, by symmetry we may conclude the same estimate with $F_1$ replaced by $F_k$
for all $1 \leq k \leq n$, where $F_k = f_k + \sum_{l \ne k} a_{kl} f_l.$
Since $\|a_{kl}\|_{L^{\infty}(U)}$ are all sufficiently small, we may write $f_1$ as a linear combination of $F_1, \dots, F_n$ with $C^{\infty}$ coefficients, say $$f_1(x) = \sum_{k=1}^n b_k(x) F_k(x)$$ with $b_k \in C^{\infty}(U)$ and conclude, as desired, that
\begin{align*}
\left|\int_{U} f_1(x) \Phi(x) dx\right| 
\leq& \sum_{k=1}^n \left|\int_{U} F_k(x) (b_k(x) \Phi(x)) dx\right| \\
\leq& \sum_{k=1}^n C \left(\|f\|_{L^1(U)} \|b_k\Phi\|_{NL_1^Q(U)} + \|g\|_{L^1(U)} \|b_k\Phi\|_{L^Q(U)}\right) \\
\leq& C \left(\|f\|_{L^1(U)} \|\Phi\|_{NL_1^Q(U)} + \|g\|_{L^1(U)} \|\Phi\|_{L^Q(U)}\right).
\end{align*}

We are left to estimate $\int_{U} F_1(x) \Phi(x) dx$ for $\Phi \in C^{\infty}_c(U)$. The argument follows closely that in \cite{MR2078071}. If $\{x_{11} = a\}$ intersects $U$ and $\lambda > 0$, we decompose $\Phi$ into $\Phi_1^a + \Phi_2^a$ as in Lemma~\ref{lem:decomp} and get
$$
\int_{\{x_{11}=a\}} F_1(x) \Phi(x) d\overline{x}
=\int_{\{x_{11}=a\}} F_1(x) \Phi_1^a(x) d\overline{x} + \int_{\{x_{11}=a\}} F_1(x) \Phi_2^a(x) d\overline{x} = I + II.
$$
The first term is bounded by 
$$
|I| \leq C \lambda^{\frac{1}{Q}} \|F_1\|_{L^1(d\overline{x})}(a) M\mathcal{I}(a) \leq C \lambda^{\frac{1}{Q}} \|f\|_{L^1(d\overline{x})}(a) M\mathcal{I}(a).
$$ 
To bound the second term, we apply the fundamental theorem of calculus along integral curves of $X_1$:
\begin{align*}
II
=& -\int_0^{\infty} \int_{\{x_{11}=a\}} \frac{d}{ds} (F_1 \Phi_2^a)(\exp(sX_1)x) d\overline{x} ds\\
=& -\int_0^{\infty} \int_{\{x_{11}=a\}} ((X_1F_1)  \Phi_2^a + F_1 (X_1\Phi_2^a))(\exp(sX_1)x) d\overline{x} ds.
\end{align*}
Using (\ref{eq:divcurlX1}), the integral of the term containing $X_1F_1$ can be written as
\begin{align*}
\int_0^{\infty} \!\! \int_{\{x_{11}=a\}} \left(\sum_{k=2}^n \left(X_k^1 (f_k \Phi_2^a) - f_k (X_k^1 \Phi_2^a) - X_k(a_{1k})f_k \Phi_2^a\right) - g \Phi_2^a \right)(\exp(sX_1)x) d\overline{x} ds
\end{align*}
The integral involving $X_k^1 (f_k \Phi_2^a)$ is bounded by $C\|f_k\|_{L^1(U)} \|\Phi_2^a\|_{L^{\infty}(U)}$, because $X_k^1 $ are parallel to all the hyperplanes $\{x_{11} = a\}$ that intersect $U$, and we can integrate by parts and bound what we obtain by changing variable $(s,\overline{x}) \mapsto \exp(sX_1)x$. Hence we can bound $II$ by
\begin{align*}
|II|
&\leq C \|f\|_{L^1(U)} \left( \|\nabla_b \Phi_2^a\|_{L^{\infty}(U)} +  \|\Phi_2^a\|_{L^{\infty}(U)} \right) + \|g\|_{L^1(U)} \|\Phi_2^a\|_{L^{\infty}(U)}\\
&\leq C \lambda^{\frac{1}{Q}-1} (\|f\|_{L^1(U)} M\mathcal{I}(a) + \|g\|_{L^1(U)} M\mathcal{J}(a)).
\end{align*}
Combining the estimates for $I$ and $II$, and optimizing $\lambda$, we get
\begin{align*}
&\left|\int_{\{x_{11}=a\}} F_1(x) \Phi(x) d\overline{x}\right|  \\
\leq & C \|f\|_{L^1(d\overline{x})}(a)^{1-\frac{1}{Q}}  M\mathcal{I}(a)^{1-\frac{1}{Q}} \left(\|f\|_{L^1(U)} M\mathcal{I}(a) + \|g\|_{L^1(U)} M\mathcal{J}(a))\right)^{\frac{1}{Q}}. 
\end{align*}
Integrating in $a$, and using Holder's inequality, we get
\begin{align*}
&\left|\int_{U} F_1(x) \Phi(x) dx\right| \\
\leq & C \|f\|_{L^1(U)}^{1-\frac{1}{Q}} \|M\mathcal{I}\|_{L^Q(dx_{11})}^{1-\frac{1}{Q}} \left(\|f\|_{L^1(U)}^{\frac{1}{Q}} \|M\mathcal{I}\|_{L^Q(dx_{11})}^{\frac{1}{Q}} + \|g\|_{L^1(U)}^{\frac{1}{Q}}\|M\mathcal{J}\|_{L^Q(dx_{11})}^{\frac{1}{Q}}\right)   \\
\leq & C \left(\|f\|_{L^1(U)} \|\Phi\|_{NL_1^Q(U)} + \|g\|_{L^1(U)} \|\Phi\|_{L^Q(U)}\right)
\end{align*}
by the boundedness of the maximal function on $L^Q(\mathbb{R})$ as desired.
\end{proof}

We now turn to the proof of Lemma~\ref{lem:decomp}. The main idea is to try to approximate the given vector fields $X_1, \dots, X_n$ on $\mathbb{R}^N$ by the left-invariant vector fields of a homogeneous group at each point. The approximation is desirable because we shall perform some convolution-like construction, and some integration by parts only work correctly if the vector fields involved are modelled on left-invariant vector fields of some group. While the approximation can be done directly in certain simple situations, in general we need to lift the vector fields $X_1, \dots, X_n$ to some vector fields $\tilde{X}_1, \dots, \tilde{X}_n$ on a higher dimensional Euclidean space $\mathbb{R}^{\tilde{N}}$, and only approximate the lifted vector fields by left-invariant vector fields. Note, however, that we cannot expect to obtain Lemma~\ref{lem:decomp} from the case of Lemma~\ref{lem:decomp} for the lifted vector fields, because the non-isotropic dimensions corresponding to the original and the lifted vector fields are different.

To begin with, let $U_1$ be a totally normal neighborhood of $0$. By Theorems 4 and 5 of Rothschild-Stein~\cite{MR0436223}, shrinking $U_1$ if necessary, there exists a neighborhood $\tilde{U_1}$ of $0$ in a higher dimensional Euclidean space $\mathbb{R}^{\tilde{N}}$, a smooth submersion $\pi \colon \tilde{U_1} \to U_1$, and smooth vector fields $\tilde{X}_1,\dots,\tilde{X}_n$ on $\tilde{U_1}$ such that
\begin{enumerate}[(a)]
\item $d\pi_{\tilde{\xi}}(\tilde{X}_k) = X_k$ for all $\tilde{\xi} \in \tilde{U_1}$ and $1 \leq k \leq n$; and
\item there exists a homogeneous group $G$ diffeomorphic to $\mathbb{R}^{\tilde{N}}$ such that
\begin{enumerate}[(i)]
\item the Lie algebra of $G$ is generated by $n$ left-invariant vector fields $Y_1,\dots,Y_n$ of degree 1, and
\item each $Y_k$ is a good approximation of $\tilde{X}_k$ at every point of $\tilde{U_1}$ in the sense we shall describe below  (see (\ref{requirement})), for $1 \leq k \leq n$.
\end{enumerate}
\end{enumerate}

In fact we shall also choose $G$ so that the grading of $\mathbb{R}^{\tilde{N}}$ at $0$ given by $\tilde{X}_1,\dots,\tilde{X}_n$ can be identified with that of the Lie algebra of $G$, in the sense that (\ref{eq:gradreq}) below holds.

Before we describe the approximation, we need to set up some notations. For each ordered tuple $\gamma = (\gamma_1, \dots, \gamma_j)$ with each $\gamma_i \in \{1, \dots, n\}$, we write $$X_{\gamma} = [X_{\gamma_1},[X_{\gamma_2},\dots,[X_{\gamma_{j-1}},X_{\gamma_j}]]].$$ Similarly for $\tilde{X}_{\gamma}$ and $Y_{\gamma}$. Remember we have defined $X_{jk}$ for $1 \leq j \leq r$, $1 \leq k \leq n_j$. For such $j$ and $k$, define now $\tilde{X}_{jk} = \tilde{X}_{\gamma}$ if $\gamma$ is some ordered tuple for which $X_{jk} = X_{\gamma}$. Any such choice of $\gamma$ will do here, and this choice will be fixed from now on. Note that $d\pi_{\tilde{\xi}}(\tilde{X}_{\gamma}) = X_{\gamma}$ for all $\tilde{\xi} \in \tilde{U_1}$, and in particular \begin{equation} \label{eq:dpi} d\pi_{\tilde{\xi}}(\tilde{X}_{jk}) = X_{jk}\end{equation} for all $\tilde{\xi} \in \tilde{U_1}$.

Now $\{\tilde{X}_{jk}\}_{1 \leq j \leq r, 1 \leq k \leq n_j}$ are linearly independent at $0$ by (\ref{eq:dpi}). We can extend this collection of vector field by choosing vectors $\tilde{X}_{jk}$, $1 \leq j \leq r$, $n_j < k \leq \tilde{n_j}$, such that each new $\tilde{X}_{jk}$ is still a commutator of $\tilde{X}_1, \dots, \tilde{X}_n$ of length $j$, and such that the extended collection of vector fields has the property that for any $1 \leq j_0 \leq r$, the restriction of $\{\tilde{X}_{jk}\}_{1 \leq j \leq j_0, 1 \leq k \leq \tilde{n_j}}$ to $0$ form a basis of the tangent subspace at $0$ spanned by the commutators of $\tilde{X}_1, \dots, \tilde{X}_n$ of length $\leq j_0$ (call this tangent subspace $\tilde{V}_{j_0}(0)$). This can be accomplished by choosing $\tilde{X}_{jk}$ inductively: first choose $\tilde{X}_{1k}$, $n_1 < k \leq \tilde{n_1}$, among the $\tilde{X}_k$'s such that $\{\tilde{X}_{1k} \colon 1 \leq k \leq \tilde{n_1}\}$ form a basis of $\tilde{V}_1(0)$. Then $\{\tilde{X}_{1k} \colon 1 \leq k \leq \tilde{n_1}\} \cup \{\tilde{X}_{2k} \colon 1 \leq k \leq n_2\}$ is a linearly independent set of vectors when restricted to $0$, because if $$\sum_{k=1}^{\tilde{n_1}} a_{1k} \tilde{X}_{1k}(0) + \sum_{k=1}^{n_2} a_{2k} \tilde{X}_{2k}(0) = 0,$$ then taking $d\pi_0$ of both sides, we get $\sum_{k=1}^{n_2} a_{2k} X_{2k}(0) \in V_1(0),$ so all $a_{2k} = 0$ by our choice of the original $X_{jk}$'s, and by linear independence of $\tilde{X}_{1k}$'s we get all $a_{1k} = 0$. Hence we can extend this collection to a basis of $\tilde{V}_2(0)$ by choosing additional $\tilde{X}_{2k}$, $n_2 < k \leq \tilde{n_2}$, that are commutators of length 2. Similarly we can choose additional $\tilde{X}_{jk}$, $n_j < k \leq \tilde{n_j}$ to satisfy the forementioned conditions. 

Since $d\pi_{\tilde{\xi}}(\tilde{X}_{jk})$ for $1 \leq j \leq r$, $n_j < k \leq \tilde{n_j}$ depends only on $\pi(\tilde{\xi})$ and not on the particular choice of $\tilde{\xi}$, we may define $X_{jk}$ on $U_1$ for such $j,k$ by (\ref{eq:dpi}) as well. By shrinking $\tilde{U_1}$, we may assume that the extended $\{\tilde{X}_{jk}\}_{1 \leq j \leq r, 1 \leq k \leq \tilde{n_j}}$ form a basis of the tangent space $T_{\tilde{\xi}}\mathbb{R}^{\tilde{N}}$ of the lifted space $\mathbb{R}^{\tilde{N}}$ at each $\tilde{\xi} \in \tilde{U_1}$, so that for each point $\tilde{\xi} \in \tilde{U_1}$, $$y \mapsto \exp(y \cdot \tilde{X}) \tilde{\xi}, \qquad y \cdot \tilde{X} := \sum_{j=1}^{r} \sum_{k=1}^{\tilde{n_j}} y_{jk} \tilde{X}_{jk}$$ defines a normal coordinate system near $\tilde{\xi}$, where $y = (y_{jk})_{1 \leq j \leq r, 1 \leq k \leq \tilde{n_j}}.$ Shrinking $\tilde{U_1}$ (hence $U_1$) if necessary, we may assume that $\tilde{U_1}$ is a totally normal neighborhood of $0$ as well.

For $1 \leq j \leq r$, $1 \leq k \leq \tilde{n_j}$, we define now $Y_{jk} = Y_{\gamma}$ if $\gamma$ is an ordered tuple for which $\tilde{X}_{jk} = \tilde{X}_{\gamma}$. Then the first claim is that we can choose $G$ such that 
\begin{align}\label{eq:gradreq}
&\text{for any $1 \leq j_0 \leq r$, $\{Y_{jk}\}_{1 \leq j \leq j_0, 1 \leq k \leq \tilde{n_j}}$ is a basis of those } \\ 
&\text{left-invariant vector fields on $G$ whose degrees are $\leq j_0$.} \notag
\end{align} 
(Note that the extended $\{\tilde{X}_{jk}\}_{1 \leq j \leq j_0, 1 \leq k \leq \tilde{n_j}}$ also satisfy an analogous condition at $0$ by our previous analysis.) 
Hence the dimension of the space of left-invariant vector fields on $G$ whose degrees are $\leq j$ is equal to $\tilde{n_1} + \dots + \tilde{n_j}$, and the homogeneous dimension of $G$ is $\tilde{Q} = \sum_{j=1}^r j\tilde{n_j}.$ Now $$y \mapsto \exp(y \cdot Y), \qquad y \cdot Y :=  \sum_{j=1}^{r} \sum_{k=1}^{\tilde{n_j}} y_{jk} Y_{jk}$$ defines a normal coordinate system on $G$, where $y = (y_{jk})_{1 \leq j \leq r, 1 \leq k \leq \tilde{n_j}}.$ On $G$ this is the only coordinate system we shall use, so we shall consistently identify $y$ with $\exp(y \cdot Y) \in G$.

Recall on $G$ we have non-isotropic dilations $$\delta \cdot y = (\delta^j y_{jk}).$$ If $\alpha = (j_1k_1, \dots, j_sk_s)$ is a multiindex, $y^{\alpha} := y_{j_1k_1}y_{j_2k_2} \dots y_{j_sk_s}$ is said to have non-isotropic degree $|\alpha|=j_1+\dots+j_s$. A function $f$ of $y$ is said to vanish to non-isotropic order $l$ at $0$ if its Taylor series expansion consists of terms whose non-isotropic degrees are all $\geq l$. A vector field $\sum_{j=1}^r \sum_{k=1}^{\tilde{n_j}} f_{jk}(y) \frac{\partial}{\partial y_{jk}}$ on $G$ is said to have local degree $\leq l$ at $y=0$ if $f_{jk}(y)$ vanish to non-isotropic orders $\geq j-l$ at $0$ for all $j$, $k$.

We can now describe the desired approximation of the lifted vector fields $\tilde{X}_k$ by $Y_k$ at every point of $\tilde{U}_1$.
Given any $\tilde{\xi} \in \tilde{U_1}$, 
$$\exp(y \cdot \tilde{X})\tilde{\xi} \mapsto y$$
defines a diffeomorphism of $\tilde{U_1}$ with a neighborhood of 0 on $G$. Any vector field $Y$ on $G$ can then be pulled back to a vector field $Y^{\tilde{\xi}}$ on $\tilde{U_1}$ using this diffeomorphism.
If we define $R_{k,\tilde{\xi}}$ to be a vector field on $G$ whose pullback $R_{k,\tilde{\xi}}^{\tilde{\xi}}$ on $\tilde{U_1}$ is given by 
\begin{equation}\label{eq:R_k}
R_{k,\tilde{\xi}}^{\tilde{\xi}} = \tilde{X}_k - Y_k^{\tilde{\xi}},
\end{equation}
then the required approximation of $\tilde{X}_k$ by $Y_k$ is the requirement that 
\begin{equation} \label{requirement}
\text{$R_{k,\tilde{\xi}}$ has local degree $\leq 0$ at $0$ for all $1 \leq k \leq n$ and all $\tilde{\xi} \in \tilde{U_1}$.}  
\end{equation}
This (and (\ref{eq:gradreq})) can be achieved if the lifted vector fields $\tilde{X}_1,\dots,\tilde{X}_n$ were free up to step $r$ and $G$ were the homogeneous group whose Lie algebra is generated by $n$ elements and free up to step $r$, but we shall not need this freeness in our argument.

If for each ordered tuple $\gamma$ and $\tilde{\xi} \in \tilde{U_1}$, we define vector fields $R_{\gamma,\tilde{\xi}}$ on $G$ by 
\begin{equation}\label{eq:R_k2}
R_{\gamma,\tilde{\xi}}^{\tilde{\xi}} = \tilde{X}_{\gamma} - Y_{\gamma}^{\tilde{\xi}}
\end{equation}
then by induction on $|\gamma|$ we can show that $R_{\gamma,\tilde{\xi}}$ has local degree $\leq |\gamma|-1$ at $0$.

Going back to $U_1 \subset \mathbb{R}^N$, recall that we defined $x \cdot X' = \sum x_{jk} X_{jk}$ for $x \in \mathbb{R}^N$ using only the vector fields $\{X_{jk}\}_{1 \leq j \leq r, 1 \leq k \leq n_j}$ that are linearly independent at $0$. We now define, for $y = (y_{jk})_{1 \leq j \leq r, 1 \leq k \leq \tilde{n_j}} \in \mathbb{R}^{\tilde{N}}$, $$y \cdot X = \sum_{j=1}^{r} \sum_{k=1}^{\tilde{n_j}} y_{jk} X_{jk}$$ using all the commutators $\{X_{jk}\}_{1 \leq j \leq r, 1 \leq k \leq \tilde{n_j}}$.

The following lemma are easy consequences of the Campbell-Hausdorff formula (see Rothschild-Stein~\cite{MR0436223}).

\begin{lemma}\label{lem:CH1}
If $S(\delta)$ is a smooth function of $\delta$ with $S(0)=s$, then $$\delta \mapsto \exp(-S(\delta)X_1) \exp(\delta X_2) \exp(sX_1) \xi$$ is a smooth curve passing through $\xi$ when $\delta = 0$, and its tangent vector at $\delta = 0$ is $$- \frac{dS}{d\delta}(0) X_1 + \sum_{j=1}^r \sum_{|\gamma|=j} s^{j-1} c_{\gamma} X_{\gamma} + \sum_{j=1}^r \sum_{|\gamma|=j} e_{\gamma,\xi}(s)X_{\gamma}$$ evaluated at $\xi$, where $c_{\gamma}$ are constants and $e_{\gamma,\xi}(s)$ are smooth functions of $s$ that vanish to order $\geq r$ at $s=0$.
\end{lemma}

\begin{lemma}\label{lem:CH2}
For any of the $X_{\gamma_0}$ with $|\gamma_0| = j_0$ and $1 \leq j_0 \leq r$, $$\delta \mapsto \exp(y \cdot X) \exp(\delta X_{\gamma_0}) \exp(-y \cdot X) \xi$$ is a smooth curve passing through $\xi$ when $\delta = 0$, and its tangent vector at $\delta = 0$ is $$\sum_{j=j_0}^r \sum_{|\gamma|=j} p_{\gamma_0,\gamma}(y) X_{\gamma} + \sum_{j=1}^r \sum_{|\gamma|=j} f_{\gamma_0,\gamma,\xi}(y) X_{\gamma}$$ evaluated at $\xi$, where $p_{\gamma_0,\gamma}(y)$ are homogeneous polynomials of $y$ of non-isotropic degrees $|\gamma|-j_0$, and $f_{\gamma_0,\gamma,\xi}(y)$ are smooth functions of $y$ that vanish to non-isotropic orders $\geq r-j_0+1$ at $y=0$.
\end{lemma}

To prove Lemma~\ref{lem:decomp}, we need one more technical lemma that allows us to integrate away the variables we added in the lifting.

\begin{lemma}\label{lem:maximal}
Shrink $U_1$ if necessary and let $\varepsilon > 0$ be sufficiently small. Let $\eta \in C^{\infty}_c(G)$, and write $$I_{\lambda}\eta(y) = \lambda^{-\tilde{Q}} \eta(\lambda^{-1} \cdot y).$$ If $\Phi \in C^{\infty}_c(U_1)$, $\xi \in \{x_{11}=a\} \cap U_1$ and $\lambda > 0$ then
$$
\int_{|y|<\varepsilon} |\Phi|(\exp(y \cdot X)\xi) |I_{\lambda}\eta(y)| dy \leq C \lambda^{\frac{1}{Q}-1} M\mathcal{J}(a)
$$
where $\mathcal{J}$ is as in Lemma~\ref{lem:decomp} and $M$ is the Hardy-Littlewood maximal function on $\mathbb{R}$.
\end{lemma}

Here $|y|$ denotes the non-isotropic norm of $y$ on $G$; i.e. $$|y| = \max_{j,k} |y_{jk}|^{1/j}.$$ Assuming these lemma for the moment, we shall complete our proof of Lemma~\ref{lem:decomp}.

\begin{proof}[Proof of Lemma~\ref{lem:decomp} continued]
We shall begin by choosing a suitable neighborhood $U$ of $0$. Let $U_1$ be a sufficiently small neighborhood of $0$ and $\varepsilon > 0$ be sufficiently small such that the previous assertions and lemma hold. Take a section $\sigma \colon U_1 \to \tilde{U_1}$ such that $\pi (\sigma (\xi)) = \xi$ for all $\xi \in U_1$. Then choose a neighborhood $U_2 \subseteq U_1$ of $0$ and reduce $\varepsilon$ if necessary such that $\exp(y\cdot \tilde{X}) \sigma(\xi) \in \tilde{U_1}$ for any $\xi \in U_2$ and any $|y| < \varepsilon$. Then it follows that $$\pi(\exp(y \cdot \tilde{X})\sigma(\xi)) = \exp(y\cdot X) \xi$$ for all such $\xi$ and $y$. This is because then the curve
\begin{align*}
[-1,1] &\to U_1\\
s &\mapsto \pi(\exp(s y\cdot \tilde{X})\sigma(\xi)) 
\end{align*}
is well-defined, and is the integral curve of $d\pi(y\cdot \tilde{X}) = y\cdot X$ beginning at $\pi(\sigma(\xi))=\xi$. The curve is thus $\exp(sy \cdot X) \xi$ for all $s \in [-1,1]$ (in particular, for $s = 1$).

We shall also apply the implicit function theorem to the equation $$\chi = \exp(sX_1) \xi$$ at the point $(\chi,s,\xi) = (0,0,0)$ and choose a neighborhood $U \Subset U_2$ of $0$ with the following property: if $I = \{a \in \mathbb{R}: \{x_{11}=a\} \cap U \ne \emptyset\}$, then for any $a \in \overline{I}$ and any point $\chi \in \overline{U}$, there is some $\xi = \xi(a,\chi) \in \{x_{11} = a\} \cap U_2$ and $s = s(a,\chi) \in (-1,1)$ such that $\chi = \exp(sX_1)\xi$. $\xi$ and $s$ will be taken to be smooth functions of $a$ and $\chi$. 

This fixes our choice of neighborhood $U$ of $0$ and a constant $\varepsilon > 0$. We now turn to construct the decomposition of $\Phi$.

Given $\Phi \in C^{\infty}_c(U)$, $a \in I$, and a parameter $\lambda > 0$, let $\eta_0 \in C^{\infty}_c(G)$ be supported on $\{|y| < \varepsilon\}$ with $\eta_0(0)=1$ with $\varepsilon > 0$ chosen as above. For any $\chi \in U$, write $\chi$ as $\chi=\exp(sX_1)\xi$ with $\xi = \xi(a,\chi)$ and $s=s(a,\chi)$ as above. Define $\Phi_2^a$ on $U$ by setting
\begin{equation} \label{eq:Phi2adef}
\Phi_2^a(\chi) = \int_{\mathbb{R}^{\tilde{N}}} \Phi(\exp(y \cdot X)\xi) I_{\sqrt{\lambda^2+s^2}}\eta_0(y)\eta_0(y) dy. 
\end{equation}

Since $\eta_0$ is supported on $\{|y|<\varepsilon\}$ and $\xi \in U_2$, in the integral $\exp(y \cdot X)\xi$ could also be written as $\pi(\exp(y \cdot \tilde{X})\tilde{\xi})$ where $\tilde{\xi} := \sigma(\xi)$. For functions $\Phi$ defined on $U_1$, we shall write $$\tilde{\Phi} = \Phi \circ \pi$$ for its pullback via $\pi$. Then $\Phi_2^a$ can also be written as $$\Phi_2^a(\chi) = \int_{\mathbb{R}^{\tilde{N}}} \tilde{\Phi}(\exp(y \cdot \tilde{X})\tilde{\xi}) I_{\sqrt{\lambda^2+s^2}}\eta_0(y) \eta_0(y) dy.$$ It follows that for $\xi \in \{x_{11}=a\} \cap U$,
\begin{equation}\label{eq:Phi1adef}
\Phi_1^a(\xi) = -\int_0^{\lambda} \int_{\mathbb{R}^{\tilde{N}}} \tilde{\Phi}(\exp(y \cdot \tilde{X})\tilde{\xi}) \frac{d}{d\lambda}I_{\lambda}\eta_0(y) \eta_0(y) dy d\lambda. 
\end{equation}
We shall estimate this as follows.

First recall that by Lemma 3.1 of \cite{MR850547}, 
$$\frac{d}{d\lambda} I_{\lambda} \eta_0(y) = \sum_{k=1}^n Y_k I_{\lambda} \eta_k (y)$$
for some functions $\eta_k \in C^{\infty}_c(G)$. For brevity of notations, in the remainder of this proof, we shall often drop the subscript $k$ in $\eta_k$ and just write $\eta$ for any function in $C^{\infty}_c(G)$. Then the inner integral in (\ref{eq:Phi1adef}) is just
\begin{align*}
-\sum_{k=1}^n \int_{\mathbb{R}^{\tilde{N}}} Y_k (\tilde{\Phi}(\exp(y \cdot \tilde{X})\tilde{\xi})) I_{\lambda}\eta(y) \eta_0(y) dy 
+ \text{errors}.
\end{align*}
The errors arise when the $Y_k$ differentiates $\eta_0(y)$ upon integration by parts; they can be estimated by 
$$C\int_{|y|<\varepsilon} |\Phi|(\exp(y \cdot X)\xi) |I_{\lambda}\eta(y)| dy,$$
and we shall call such terms acceptable errors. To tackle the main term, note that 
\begin{align*}
Y_k (\tilde{\Phi}(\exp(y \cdot \tilde{X})\tilde{\xi})) 
&= (Y_k^{\tilde{\xi}} \tilde{\Phi})(\exp(y \cdot \tilde{X})\tilde{\xi})\\
&= (\tilde{X_k} \tilde{\Phi})(\exp(y \cdot \tilde{X})\tilde{\xi}) + (R_{k,\tilde{\xi}}^{\tilde{\xi}} \tilde{\Phi})(\exp(y \cdot \tilde{X})\tilde{\xi})\\
&= (\tilde{X_k} \tilde{\Phi})(\exp(y \cdot \tilde{X})\tilde{\xi}) + R_{k,\tilde{\xi}} (\tilde{\Phi}(\exp(y \cdot \tilde{X})\tilde{\xi}))
\end{align*}
where the vector fields $R_{k,\tilde{\xi}}$ have local degrees $\leq 0$ at $0$. If $R$ is a vector field on $G$ that has local degree $\leq l$ at $0$, then 
\begin{equation} \label{eq:remain}
|R I_{\lambda} \eta(y)| \leq C \lambda^{-l} |I_{\lambda}\eta'(y)|
\end{equation} 
when $|y| < \varepsilon$. Here $\eta'$ is just some function in $C^{\infty}_c(G)$, and again we shall just write $\eta$ for $\eta'$. Hence the integral of the terms involving $R_{k,\tilde{\xi}}$ contributes only acceptable errors upon integration by parts. To deal with the terms involving $\tilde{X_k} \tilde{\Phi}$, we observe that $(\tilde{X}_{\gamma} \tilde{\Phi}) (\tilde{\chi}) = (X_{\gamma} \Phi)(\pi(\tilde{\chi}))$ for all $\tilde{\chi} \in \tilde{U_1}$. Hence 
\begin{equation} \label{eq:lift}
(\tilde{X}_{\gamma} \tilde{\Phi}) (\exp(y \cdot \tilde{X})\tilde{\xi}) = (X_{\gamma} \Phi) (\exp(y \cdot X)\xi) 
\end{equation}
for $\xi \in U_2$ and $|y|<\varepsilon$,
and in particular $(\tilde{X_k} \tilde{\Phi})(\exp(y \cdot \tilde{X})\tilde{\xi}) = (X_k \Phi)(\exp(y \cdot X)\xi).$ Putting everything together, 
\begin{equation}\label{eq:Phi1}
|\Phi_1^a(\xi)| \leq C \int_0^{\lambda} \int_{|y|<\varepsilon} (|\nabla_b \Phi| + |\Phi|) (\exp(y \cdot X)\xi) |I_{\lambda}\eta(y)| dy d\lambda.
\end{equation}

Next we estimate $|\nabla_b \Phi_2^a(\chi)|$ for $\chi \in U$. Write $\chi = \exp(s X_1) \xi$ with $\xi = \xi(a,\chi)$ and $s=s(a,\chi)$. Then
\begin{align*}
(X_1 \Phi_2^a)(\chi) 
&= \int_{\mathbb{R}^{\tilde{N}}} \tilde{\Phi}(\exp(y \cdot \tilde{X})\tilde{\xi}) \left. \frac{d}{d\delta} \right|_{\delta = \sqrt{\lambda^2 + s^2}} I_{\delta}\eta_0(y) \eta_0(y) dy \frac{s}{\sqrt{\lambda^2+s^2}}.
\end{align*}
Note that $\frac{s}{\sqrt{\lambda^2+s^2}} \leq 1$. Arguing as before we get 
\begin{equation}\label{eq:Phi2d1}
|(X_1 \Phi_2^a) (\chi)| \leq C \int_{|y|<\varepsilon} (|\nabla_b \Phi| + |\Phi|) (\exp(y \cdot X)\xi) |I_{\sqrt{\lambda^2 + s^2}}\eta(y)| dy.
\end{equation}

To estimate $(X_2\Phi_2^a)(\chi)$, note that  
$$
(X_2 \Phi_2^a)(\chi) = \left. \frac{d}{d\delta} \right|_{\delta = 0} \Phi_2^a (\exp(\delta X_2) \exp(sX_1) \xi).
$$
Given $\delta$ close to 0, 
choose $S(\delta)=S(\delta,a,\chi)$ such that 
$$
\exp(-S(\delta)X_1) \exp(\delta X_2) \exp(sX_1) \xi \in \{x_{11}=a\}\cap U_2.
$$ 
This is a smooth function of $\delta$, $a$ and $\chi$ with $S(0) = s$. Then $(X_2 \Phi_2^a)(\chi)$ is given by
\begin{align*}
\left. \frac{d}{d\delta} \right|_{\delta = 0} \int_{\mathbb{R}^{\tilde{N}}} \! \Phi(\exp(y \cdot X)\exp(-S(\delta)X_1) \exp(\delta X_2) \exp(sX_1) \xi) I_{\sqrt{\lambda^2+S(\delta)^2}}\eta_0(y) \eta_0(y) dy. 
\end{align*}
If the derivative fall on $I_{\sqrt{\lambda^2+S(\delta)^2}}\eta_0$, then we get 
$$
\int_{\mathbb{R}^{\tilde{N}}} \Phi(\exp(y \cdot X) \xi) \left. \frac{d}{d\delta} \right|_{\delta = \sqrt{\lambda^2 + s^2}} I_{\delta}\eta_0(y) \eta_0(y) dy \frac{s}{\sqrt{\lambda^2+s^2}} \frac{dS}{d\delta}(0,a,\chi)
$$
and this is bounded by 
\begin{equation}\label{eq:bdd2}
C \int_{|y|<\varepsilon} (|\nabla_b \Phi| + |\Phi|) (\exp(y \cdot X)\xi) |I_{\sqrt{\lambda^2 + s^2}}\eta(y)| dy
\end{equation}
just as before, because $\frac{dS}{d\delta}(0,a,\chi)$ is bounded for $a \in \overline{I}$ and $\chi \in \overline{U}$. If the derivative fall on $\Phi$, we shall invoke Lemma~\ref{lem:CH1} to show that the same bound holds for the integral. In fact
\begin{align*}
&  \left. \frac{d}{d\delta} \right|_{\delta = 0} \Phi(\exp(y \cdot X)\exp(-S(\delta)X_1) \exp(\delta X_2)  \exp(sX_1) \xi) \\
=& - \frac{dS}{d\delta}(0,a,\chi)  \left. \frac{d}{d\delta} \right|_{\delta = 0} \Phi(\exp(y \cdot X) \exp(\delta X_1)  \xi)  \\
&\quad +\sum_{j=1}^r \sum_{|\gamma|=j} s^{j-1} c_{\gamma} \left. \frac{d}{d\delta} \right|_{\delta = 0} \Phi(\exp(y \cdot X) \exp(\delta X_{\gamma}) \xi)  \\
&\quad +\sum_{j=1}^r \sum_{|\gamma|=j} e_{\gamma,\xi}(s) \left. \frac{d}{d\delta} \right|_{\delta = 0} \Phi(\exp(y \cdot X) \exp(\delta X_{\gamma}) \xi)\\
=& I + II + III.
\end{align*}
The contribution of $I$ in the integral can be absorbed into that of $II$ once we note that $\frac{dS}{d\delta}(0,a,\chi)$ is bounded. To estimate of the contribution of $II$, fix $\gamma_0$ with $|\gamma_0| = j_0$ and consider 
\begin{equation} \label{eq:integralX_jk}
\int_{\mathbb{R}^{\tilde{N}}} s^{j_0-1} \left. \frac{d}{d\delta} \right|_{\delta = 0} \Phi(\exp(y \cdot X) \exp(\delta X_{\gamma_0}) \xi) I_{\sqrt{\lambda^2+s^2}}\eta_0(y) \eta_0(y) dy.
\end{equation}
Using Lemma~\ref{lem:CH2}, the derivative inside this integral is equal to
\begin{align}
\sum_{j=j_0}^r \sum_{|\gamma|=j} p_{\gamma_0, \gamma}(y) (X_{\gamma}\Phi)(\exp(y \cdot X) \xi) + \sum_{j=1}^r \sum_{|\gamma|=j} f_{\gamma_0,\gamma,\xi}(y) (X_{\gamma}\Phi)(\exp(y \cdot X) \xi) \label{eq:pointwiseX_jk}
\end{align}
where $p_{\gamma_0,\gamma}(y)$ are homogeneous polynomials of non-isotropic degrees $|\gamma|-j_0$ and $f_{\gamma_0,\gamma,\xi}(y)$ vanishes to non-isotropic orders $\geq r-j_0+1$ at $y=0$. The second term is just a sum of
\begin{align*}
f_{\gamma_0,\gamma,\xi}(y)(X_{\gamma}\Phi)(\exp(y \cdot X) \xi) 
=& f_{\gamma_0,\gamma,\xi}(y) (Y_{\gamma}+ R_{\gamma,\tilde{\xi}})(\tilde{\Phi}(\exp(y \cdot \tilde{X}) \tilde{\xi})) 
\end{align*}
and the vector field $f_{\gamma_0,\gamma,\xi}(y) (Y_{\gamma}+ R_{\gamma,\tilde{\xi}})$ has local degree $\leq j_0-1$ at $y=0$. Hence by (\ref{eq:remain}), the contribution of this term in (\ref{eq:integralX_jk}) is bounded by
$$ \sum_{i=0}^{j_0-1} \frac{s^{j_0-1}}{(\lambda^2+s^2)^{i/2}} \int_{|y|<\varepsilon} |\Phi|(\exp(y \cdot X) \xi) |I_{\sqrt{\lambda^2+s^2}}\eta(y)| dy$$
upon integration by parts. But since now $|s|<\varepsilon$, the contribution of this term is bounded by 
$$C\int_{|y|<\varepsilon} |\Phi|(\exp(y \cdot X)\xi) |I_{\sqrt{\lambda^2+s^2}}\eta(y)| dy.$$
We shall also call such terms acceptable errors. Now to deal with the first term in (\ref{eq:pointwiseX_jk}), observe that by (\ref{eq:lift}) and (\ref{eq:R_k2}),
\begin{align*}
p_{\gamma_0,\gamma}(y) (X_{\gamma}\Phi)(\exp(y \cdot X) \xi)
&= p_{\gamma_0,\gamma}(y) (Y_{\gamma} + R_{\gamma,\tilde{\xi}}) (\tilde{\Phi}(\exp(y \cdot \tilde{X}) \tilde{\xi})) 
\end{align*}
where $R_{\gamma,\tilde{\xi}}$ has local degree $\leq |\gamma|-1$ at $0$. It follows that $p_{\gamma_0,\gamma}(y)R_{\gamma,\tilde{\xi}}$ has local degree $\leq j_0-1$ at $0$, and the terms in the integral (\ref{eq:integralX_jk}) that involves $p_{\gamma_0,\gamma}(y) R_{\gamma,\tilde{\xi}}$ is an acceptable error just as above. We are left with estimating
$$
s^{j_0-1} \int_{\mathbb{R}^{\tilde{N}}} \sum_{j=j_0}^r \sum_{|\gamma|=j} p_{\gamma_0,\gamma}(y) Y_{\gamma} (\tilde{\Phi}(\exp(y \cdot \tilde{X}) \tilde{\xi})) I_{\sqrt{\lambda^2+s^2}}\eta_0(y) \eta_0(y) dy.
$$
If we write each $Y_{\gamma}$ as a commutator of length $|\gamma|$ and integrate by parts $|\gamma|-1$ times, we get
$$\int_{\mathbb{R}^{\tilde{N}}} \sum_{k=1}^n \sum_{l=0}^{r-j_0} \sum_{i=0}^{j_0+l-1} \frac{s^{j_0-1} q_l(y)}{(\lambda^2+s^2)^{i/2}}(Y_{k}^{\tilde{\xi}}\tilde{\Phi})(\exp(y \cdot \tilde{X}) \tilde{\xi}) I_{\sqrt{\lambda^2+s^2}}\eta(y) \eta(y) dy$$
where $q_l(y)$ is a homogeneous polynomial of non-isotropic degree $l$. Approximating $Y_k^{\tilde{\xi}}$ by $X_k$ as in (\ref{eq:R_k}), and noting that $$\frac{s^{j_0-1} q_l(y)}{(\lambda^2+s^2)^{(j_0+l-1)/2}} I_{\sqrt{\lambda^2+s^2}}\eta(y) = \frac{s^{j_0-1}}{(\lambda^2+s^2)^{(j_0-1)/2}} I_{\sqrt{\lambda^2+s^2}}\eta'(y) \leq I_{\sqrt{\lambda^2+s^2}}\eta'(y)  $$ for some $\eta' \in C^{\infty}_c(G)$, we see that the terms with $i = j_0+l-1$ are bounded by (\ref{eq:bdd2}) as well. Integrating by parts in $Y_k$, we see that the terms with $i < j_0+l-1$ are acceptable errors. Altogether, summing over $\gamma_0$ in (\ref{eq:integralX_jk}), we see that the contribution of $II$ in the original integral is controlled by (\ref{eq:bdd2}). Finally, in a similar way, we conclude that the contribution of $III$ is only an acceptable error, and hence 
\begin{equation}\label{eq:Phi2d2}
|(X_2\Phi_2^a)(\chi)| \leq C\int_{|y|<\varepsilon} (|\nabla_b\Phi|+|\Phi|)(\exp(y \cdot X)\xi) |I_{\sqrt{\lambda^2+s^2}}\eta(y)| dy. 
\end{equation}
Similarly we have the same estimate for $X_k\Phi_2^a$ for $2 \leq k \leq n$.

It is now easy to complete the proof of Lemma~\ref{lem:decomp} by appealing to Lemma~\ref{lem:maximal}, using (\ref{eq:Phi1}), (\ref{eq:Phi2d1}), (\ref{eq:Phi2d2}) and (\ref{eq:Phi2adef}). 
\end{proof}

We shall now prove the Lemma we used in proving Lemma~\ref{lem:decomp}.

\begin{proof}[Proof of Lemma~\ref{lem:maximal}]
For $y \in \mathbb{R}^{\tilde{N}}$, we write $y=(y',y'')$ where 
$$
y'=(y_{jk})_{1 \leq j \leq r, 1 \leq k \leq n_j} \in \mathbb{R}^N
$$ 
and $y''$ denote the remaining variables. We shall also write $y'=(y_1',\dots,y_r')$ where $y_j'= (y_{jk})_{1 \leq k \leq n_j}$, and introduce the shorthand $y^{(j)} = (y_1',\dots,y_j')$. Define non-isotropic norms $|y'|=|(y',0)|$, $|y''|=|(0,y'')|$ and $|y^{(j)}| = |(y^{(j)},0,\dots,0)|$. For $x \in \mathbb{R}^N$ sufficiently close to $0$, we shall also write $x = (x_1, \dots, x_r)$ with $x_j = (x_{jk})_{1 \leq k \leq n_j}$ and define non-isotropic norms $|x_j| = \max_{1 \leq k \leq n_j} |x_{jk}|^{1/j}$ for each $j$. 

Now for any given $\xi \in U_1$ and $y'' \in \mathbb{R}^{\tilde{N}-N}$, consider the map $y' \mapsto x(\xi,y) \in \mathbb{R}^N$ where $x(\xi,y)$ is defined by 
\begin{equation}\label{eq:u} 
x(\xi,y) = \exp(y \cdot X)\xi.
\end{equation}
By shrinking $U_1$ if necessary and taking $\varepsilon > 0$ to be sufficiently small, according to the inverse function theorem, for any $y''$ with $|y''|<\varepsilon$ and for any $\xi \in U_1$, the map is a diffeomorphism from the set $\{|y'|<\varepsilon\}$ to a neighborhood of $0$ in $\mathbb{R}^N$.
By the Campbell-Hausdorff formula, if $h_{jk}$ denote the coordinate function in the normal coordinates at $0$, i.e. $h_{jk}(x) = x_{jk}$ in $U_1$, then for all $|y|<\varepsilon$ and $\xi \in U_1$, 
$$
h_{jk}(x(\xi,y)) = \sum_{p=0}^r \frac{1}{p!}((y \cdot X)^p h_{jk})(\xi) + O(|y|^{r+1}).
$$ 
Here $O(|y|^j)$ shall always denote a term $\leq C|y|^j$ with $C$ independent of $\xi$. It follows that 
\begin{equation}\label{eq:changevar}
h_{jk}(x(\xi,y)) = h_{jk}(\xi) + \sum_{1 \leq |\alpha| \leq r} a_{jk,\alpha}(\xi) y^{\alpha} + O(|y|^{r+1})
\end{equation}
for some smooth functions $a_{jk,\alpha}$ of $\xi$. For each $1 \leq j_0 \leq r$, define $g_{jk,\xi,y''}^{(j_0)}(y^{(j_0)})$ to be the sum of the terms on the right hand side of the above equation whose non-isotropic degrees in $y$ are $\leq j_0$ (note that this depends only on $y^{(j_0)}$ and $y''$ but not $y_{j_0+1}',\dots,y_r'$) and let $g_{\xi,y''}^{(j_0)}$ be the map $$y^{(j_0)} \mapsto \left(g_{jk,\xi,y''}^{(j_0)}(y^{(j_0)})\right)_{1 \leq j \leq j_0, 1 \leq k \leq n_j}.$$ 
By shrinking $U_1$ and decreasing $\varepsilon$ again if necessary, using the inverse function theorem, we may assume that for any $\xi \in U_1$, $|y''|<\varepsilon$, and $1 \leq j \leq r$, the map $g_{\xi,y''}^{(j)}$ is a diffeomorphism from the set $\{|y^{(j)}|<\varepsilon\}$ to its image. By taking $\varepsilon$ sufficiently small, we may also assume that for all such $\xi$, $y''$ and $j$, 
\begin{equation}\label{eq:comparable}
|g_{\xi,y''}^{(j)}(y_1^{(j)})-g_{\xi,y''}^{(j)}(y_2^{(j)})| \simeq |y_1^{(j)}-y_2^{(j)}|
\end{equation} 
for all $|y_1^{(j)}|, |y_2^{(j)}| < \varepsilon$, with implicit constants independent of $\xi$ and $y''$. We may also take some $\varepsilon_1 < \varepsilon$ so that for all such $\xi$, $y''$ and $j$, the $\varepsilon_1$-neighborhood of the image of $\{|y'|<\varepsilon_1\}$ under $g_{\xi,y''}^{(j)}$ is contained in the image of $\{|y'|<\varepsilon\}$ under the same map. 

We claim that there exists a small constant $c < 1$ such that for all $\lambda < c\varepsilon$, $|y''|<\lambda$ and $\xi \in U_1$, the map $$y' \mapsto x(\xi,y)$$ in (\ref{eq:u}) maps the set $\{|y'|<\lambda\}$ into the set $$S_{\lambda}:=\{|x_1-h_1(\xi)|, |x_2-f_{1,\xi,y''}(x_1)|, \dots, |x_r - f_{r-1,\xi,y''}(x_1,\dots,x_{r-1})| \leq C\lambda\}$$ for some smooth functions $f_{j,\xi,y''}$ of $(x_1,\dots,x_{j-1})$, where $h_1 = (h_{1k})_{1\leq k \leq n_1}$ are the first $n_1$ coordinate functions in normal coordinates at 0 and $C$ is a constant that does not depend on $\xi$, $y''$ and $\lambda$. The lemma follows from the claim: for $\xi \in \{x_{11}=a\} \cap U_1$, if $\lambda < c\varepsilon$, we can make a change of variable $y' \mapsto x = x(\xi,y)$ in the integral to be estimated and bound that by
\begin{align*}
C\lambda^{-\tilde{Q}} \int_{|y''|<\lambda} \int_{|x_1-h_1(\xi)| \leq C\lambda} \cdots \int_{|x_r - f_{r-1,\xi,y''}(x_1,\dots,x_{r-1})| \leq C\lambda} |\Phi|(x) dx_r \dots dx_1 dy''
\end{align*}
because the Jacobian of the change of variable $J(\xi,x,y'')$ is uniformly bounded in $\xi$, $x$ and $y''$. Using Holder's inequality successively, this is bounded by
\begin{align*}
&C \lambda^{-\tilde{Q}} \int_{|y''|<\lambda} \int_{|x_{11}-a|\leq C\lambda} \left(\int_{\overline{x} \in \mathbb{R}^{N-1}} |\Phi|^Q(x) d\overline{x} \right)^{\frac{1}{Q}}  \left(\lambda^{Q-1}\right)^{\frac{Q-1}{Q}} dx_{11} dy'' \\
\leq & C \lambda^{-\tilde{Q}} \lambda^{\tilde{Q}-Q} \left(\lambda^{Q-1}\right)^{\frac{Q-1}{Q}} \lambda M\mathcal{J}(a) \\
\leq & C \lambda^{\frac{1}{Q}-1} M\mathcal{J}(a)
\end{align*}
where $\mathcal{J}$ is as in the Lemma and $M$ is the standard Hardy-Littlewood maximal function on $\mathbb{R}$. If $\lambda > c \varepsilon$, the estimate is only easier. Therefore it remains to prove the claim.

Let $\lambda < c\varepsilon$, $|y|<\lambda$ and $\xi \in U_1$. Let $x = x(\xi,y)$. We shall show that $x \in S_{\lambda}$. Write $x = (x_1, \dots, x_r)$, $x^{(j)} = (x_1, \dots, x_j)$ as we did for $y'$. First, from (\ref{eq:changevar}), 
$$
x_1 = h_1(\xi) + O(\lambda)
$$
with implicit constant independent of $\xi$ and $y$. Next, for $1 \leq j < r$, by definition of $g^{(j)}_{\xi,y''}$, 
$$
x^{(j)} = g^{(j)}_{\xi,y''}(y^{(j)}) + O(\lambda^{j+1}).
$$ 
Since $\lambda < c\varepsilon$, by taking $c$ sufficiently small, the term $O(\lambda^{j+1})$ can be made smaller than $\varepsilon_1$. By our choice of $\varepsilon_1$, $x^{(j)}$ is thus in the image of $\{|y'|<\varepsilon\}$ under $g_{\xi,y''}^{(j_0)}$. 
As a result, by (\ref{eq:comparable}),  
$$
y^{(j)} = G^{(j)}_{\xi,y''}(x^{(j)}) + O(\lambda^{j+1})
$$
with implicit constants independent of $\xi$ and $y''$,
where $G^{(j)}_{\xi,y''}$ is the inverse of the function $g^{(j)}_{\xi,y''}$. Hence by (\ref{eq:changevar}) again, looking only at terms of non-isotropic degrees $\leq j$ and substituting $y^{(j)}$ for $G^{(j)}_{\xi,y''}(x^{(j)})+ O(\lambda^{j+1})$, we get
$$
x_{j+1} = f_{j,\xi,y''}(x^{(j)}) + O(\lambda^{j+1})
$$
for some function $f_{j,\xi,y''}$ of $x^{(j)}$, with implicit constant independent of $\xi$ and $y''$. Hence $x \in S_{\lambda}$, and this completes the proof of the claim.
\end{proof}

\begin{proof}[Proof of Lemma~\ref{lem:CH2}]
Fix $\gamma_0$ with $|\gamma_0| = j_0$ and  $1 \leq j_0 \leq r$. Let $\phi$ be any smooth function near $\xi$. The Taylor expansion of the function $$\phi(\exp(y \cdot X) \exp(\delta X_{\gamma_0}) \exp(-y \cdot X) \xi)$$ around $y=0$ and $\delta = 0$ is given by
$$\sum_{j=0}^{r-1} \frac{1}{j!}(-y \cdot X)^j \sum_{k=0}^1 \frac{1}{k!}(\delta X_{\gamma_0})^k \sum_{l=0}^{r-1} \frac{1}{l!}(y \cdot X)^l \phi(\xi) + O(|y|^{r}, \delta^2).$$ By the Campbell-Hausdorff formula, this is equal to 
\begin{align*}
\left(\sum_{i=0}^{r-1} \frac{1}{i!} \left(\delta X_{\gamma_0} + \delta \sum_{j=1}^{r-1} d_j ad(-y \cdot X)^j X_{\gamma_0} \right)^i \phi\right) (\xi) + O(|y|^{r},\delta^2)
\end{align*}
where $c_j$ and $d_j$ are absolute constants. Differentiating in $\delta$ and evaluating at $\delta=0$, we get 
$$\left(X_{\gamma_0} + \sum_{j=1}^{r-1} d_j ad(-y \cdot X)^j X_{\gamma_0}\right)\phi(\xi)+O(|y|^{r}).$$ Since $\phi$ is arbitrary, the tangent vector of the curve in the lemma is given by
$$X_{\gamma_0} + \sum_{j=1}^{r-1} d_j ad(-y \cdot X)^j X_{\gamma_0}+O(|y|^{r})$$
around $y=0$, which has the desired form.
\end{proof}

\begin{proof}[Proof of Lemma~\ref{lem:CH1}]
For each small $s$, let $\gamma_s(\delta)$ be the curve $$ \gamma_s(\delta)= \exp(-sX_1) \exp(\delta X_2) \exp(sX_1) \xi $$ with $\gamma_s(0)=\xi$. Its tangent vector at $\delta = 0$ can be calculated by the Campbell-Hausdorff formula as in the proof of Lemma~\ref{lem:CH2}: in fact $$\gamma_s'(0)=X_2 + \sum_{j=1}^{r-1} d_j s^j ad(X_1)^j X_2 + O(|s|^{r}).$$ Hence the tangent vector of the curve in the lemma at $\delta = 0$ is 
$$ - \frac{dS}{d\delta}(0) X_1 + \gamma_s'(0) = - \frac{dS}{d\delta}(0) X_1 + X_2 + \sum_{j=1}^{r-1} d_j s^j ad(X_1)^j X_2 + O(|s|^{r})$$
which has the desired form.
\end{proof}

\section{Gagliardo-Nirenberg inequality for $\dbarb$}

We are now ready to prove our $L^1$ estimates for $\dbarb$. The proof is by duality as in \cite{MR2122730}. The new ingredient here is a localization to small coordinate patches where Theorem~\ref{thm:PLA} applies, with $X_1, \dots, X_{2n}$ being the real and imaginary parts of the anti-holomorphic vector fields $\Zbar_1, \dots, \Zbar_n$. We also need to use the regularity on $L^Q$ of the relative fundamental solutions of $\dbarb$, $\dbarb^*$ and $\boxb$; this is provided by the result of Koenig \cite{MR1879002} on maximal subellipticity when $M$ is of finite commutator type and satisfies condition $D(q_0)$, and by classical results when $M$ satisfies condition $Y(q)$.

\begin{proof}[Proof of Theorem~\ref{thm:PLB}]  
To prove (a), let $u$ be a smooth $(0,q)$-form be orthogonal to the kernel of $\boxb$ where $q_0 \leq q \leq n-q_0$ and $q \ne 1$ nor $n-1$. By duality, it suffices to prove that $$\left|\langle u, \phi \rangle\right| \leq C \left(\|\dbarb u\|_{L^1(M)} + \|\dbarb^* u\|_{L^1(M)}\right) \|\phi\|_{L^Q(M)}$$ for all smooth $(0,q_0)$-forms $\phi$ where $Q = 2n + m$. To do so, note that by Hodge decomposition,
$$
\langle u, \phi \rangle = \langle \dbarb u, \dbarb K_{q} \phi \rangle + \langle \dbarb^* u, \dbarb^* K_{q} \phi \rangle
$$ 
where $K_q$ is the relative solution operator for $\boxb$ on $(0,q)$ forms.
To estimate this, recall that near each point, there is a neighborhood $U$ on which a local frame of holomorphic tangent vectors $Z_1,\dots,Z_n$ is defined, and that the conclusion of Theorem~\ref{thm:PLA} holds for $\phi$ supported there. Since $M$ is compact, we can cover it by finitely many such charts, and let $\sum_{\alpha} \eta_{\alpha}^2 = 1$ be a partition of unity subordinate to it. We shall estimate $\langle \eta_{\alpha} \dbarb u, \eta_{\alpha} \dbarb K_{q}\phi \rangle$ for each $\alpha$: Let $\omega_1, \dots, \omega_n$ be a dual frame of (0,1) forms to $Z_1, \dots, Z_n$ on the support of $\eta_{\alpha}$, and write $$\dbarb u = \sum_{|I|=q+1} (\dbarb u)_I \overline{\omega}^I$$ there. Since $q \ne n-1$, either $q = n$ in which case $\dbarb u = 0$ and we have a trivial estimate, or $\dbarb u$ is a $(0,q+1)$ form with $q+1<n$, so for each multiindex $I$ with $|I| = q+1$, there exists an index $j$ that does not appear in $I$. Since $\dbarb \dbarb u = 0$, on the support of $\eta_{\alpha}$ we have
$$ \overline{Z}_j (\dbarb u)_I = \sum_{k \in I} \pm \overline{Z}_k (\dbarb u)_{jI_k} + O(\dbarb u)$$
where $I_k$ is $I$ with $k$ removed, and $O(\dbarb u)$ represent terms that are 0th order in components of $\dbarb u$. Now write $$\overline{Z}_j = X_j + i X_{n+j}$$ where $X_j$ and $X_{n+j}$ are the real and imaginary parts of $\overline{Z}_j$ respectively. Then 
$$
X_j (\eta_{\alpha} \dbarb u)_I + X_{n+j}  (i\eta_{\alpha} \dbarb u)_I + \sum_{k \in I} \pm X_k (\eta_{\alpha}\dbarb u)_{jI_k} \pm X_{n+k} (i\eta_{\alpha}\dbarb u)_{jI_k} = O(\dbarb u).
$$
Note that at any point, the non-isotropic dimension attached to the real vector fields $X_1, \dots, X_{2n}$ is at most $Q = 2n+m$, because the missing direction $iT$ can be generated by at most $m$ brackets of these vector fields (in other words, in the notations of Theorem~\ref{thm:PLA}, $n_1 = 2n$ and $n_{j_0} = 1$ for some $2 \leq j_0 \leq m$, with all other $n_j$ being zero). 
Since the conclusion of Theorem~\ref{thm:PLA} holds in the support of $\eta_{\alpha}$, we have
\begin{align*}
\left| \int_M \eta_{\alpha} (\dbarb u)_I \overline{\eta_{\alpha} (\dbarb K_{q} \phi)_I} d\text{vol}_g \right| 
&\leq C \|\dbarb u \|_{L^1(M)} \|\eta_{\alpha} \dbarb K_{q}\phi \|_{NL_1^Q(M)}\\ 
&\leq C \|\dbarb u \|_{L^1(M)} \|\phi\|_{L^Q(M)},
\end{align*}
the last estimate following from Theorem 5.12 of Koenig~\cite{MR1879002} on the regularity of $K_q$. Here we used the facts that $M$ is pseudoconvex CR manifold of real dimension $\geq 5$, that $\dbarb$ has closed ranges on $L^2$ on all forms, that $M$ is of finite commutator type and that $M$ satisfies condition $D(q_0)$. This proves the desired estimate for $|\langle \dbarb u, \dbarb K_{q}\phi \rangle|$, and a similar calculation establishes the desired estimate for $|\langle \dbarb^* u, \dbarb^* K_{q}\phi \rangle|$.

Similarly, to prove (b), if $v$ is a smooth $(0,q_0-1)$ form on $M$ orthogonal to the kernel of $\dbarb$, then $$\langle v, \phi\rangle = \langle \dbarb v, G_{q_0-1}' \phi\rangle$$ for all smooth $(0,q_0-1)$ forms $\phi$, where $G_{q_0-1}'$ is the relative fundamental solution of $\dbarb^* \colon L^2(\Lambda^{0,q_0}) \to L^2(\Lambda^{0,q_0-1})$. Note now $q_0 \leq \frac{n}{2} < n$, so $q_0 - 1 \ne n-1$ and $\dbarb v$ is not a top form. It follows that every component of $\dbarb v$ satisfies some divergence type condition as above. Using an argument similar to the one above, and Corollary 5.13 of Koenig~\cite{MR1879002} on the regularity of $G_{q_0-1}'$ under our assumptions on $M$, we then get $$|\langle v, \phi\rangle| \leq C \|\dbarb v \|_{L^1(M)} \|\phi\|_{L^Q(M)},$$ and the desired estimate follows.

The proof of (c) is similar to (b), except that we write, for smooth $(0,n-q_0+1)$ form $w$ orthogonal to the kernel of $\dbarb^*$, that $$\langle w, \phi\rangle = \langle \dbarb^* w, G_{n-q_0} \phi\rangle$$ for all smooth $(0,n-q_0+1)$ forms $\phi$, where $G_{n-q_0}$ is the relative fundamental solution to $\dbarb \colon L^2(\Lambda^{0,n-q_0}) \to L^2(\Lambda^{0,n-q_0+1})$, and use that $\dbarb^* w$ is not a function instead. The required regularity for $G_{n-q_0}$ is again guaranteed by Corollary 5.13 of Koenig~\cite{MR1879002}.
\end{proof}

\begin{proof}[Proof of Theorem~\ref{thm:PLC}]
The proof is very similar to Theorem~\ref{thm:PLB} above. In fact the key ingredients to the proof of Theorem~\ref{thm:PLB} are the Hodge decompositions for $\boxb$, $\dbarb$ and $\dbarb^*$, and the corresponding maximal subelliptic estimates as given by the theorem of Koenig. We have all these when $M$ satisfies condition $Y(q)$ instead; in fact then $\dbarb$ satisfies a subelliptic $\frac{1}{2}$ estimate and $K_q$ gains 2 derivatives in the good directions (see Folland-Kohn \cite{MR0461588} and Rothschild-Stein \cite{MR0436223}). The details of the proof are omitted.
\end{proof}

\section*{Acknowledgements}

I would like to thank my adviser E. M. Stein for suggesting this problem, and for his constant support throughout the project. 

\bibliographystyle{mrl}
\bibliography{div_curl_paper}

\end{document}